\documentclass[11pt]{amsart}
\usepackage{amsmath,amsthm,amssymb,xypic,array}
\usepackage[T1]{fontenc}
\usepackage{amsmath}
\usepackage{amssymb}
\usepackage{amsthm}
\usepackage{setspace}
\usepackage[cp850]{inputenc}
\usepackage[bookmarksnumbered,plainpages,hypertex]{hyperref}
\usepackage{graphicx}
\usepackage{fancyhdr}
\usepackage{tikz}
\usepackage{times}
\usepackage{color}

\providecommand{\U}[1]{\protect\rule{.1in}{.1in}}
\providecommand{\U}[1]{\protect\rule{.1in}{.1in}}
\providecommand{\U}[1]{\protect\rule{.1in}{.1in}}
\providecommand{\U}[1]{\protect\rule{.1in}{.1in}}
\providecommand{\U}[1]{\protect\rule{.1in}{.1in}}
\input{xy}
\xyoption{all}
\theoremstyle{theorem}

\newtheorem{Theorem}{Theorem}[section]
\newtheorem*{theoremn}{Theorem}

\newtheorem{Proposition}[Theorem]{Proposition}
\newtheorem{Corollary}[Theorem]{Corollary}

\theoremstyle{definition}

\newtheorem{Definition}[Theorem]{Definition}
\newtheorem{Remark}[Theorem]{Remark}

\numberwithin{equation}{section}

\newcommand{\Pic}{\operatorname{Pic}}

\newcommand{\p}{{\mathbb P}}
\renewcommand{\a}{{\`a}}

\newcommand{\Mor}{\operatorname{Mor}}

\def\leq{\leqslant}
\def\geq{\geqslant}

\def\bibaut#1{{\sc #1}}

\def\phi{\varphi}
\def\ro[#1]{{\textcolor{red}{#1}}}

\renewcommand{\a}{\`a }

\newcommand{\arXiv}[1]{\href{http://arxiv.org/abs/#1}{arXiv:#1}}

\begin{document}

\begin{abstract}
In this paper we give for any integer $l\geq 2$ a numerical criterion ensuring the existence of a chain of length $l$ of lines through two general points of an irreducible variety $X \subset \mathbb{P}^{N}$, involving the degrees and the number of homogeneous polynomials defining $X$. We show that our criterion is sharp.
\end{abstract}

\title[Varieties connected by chains of lines]{Varieties connected by chains of lines}
\author[Simone Marchesi]{Simone Marchesi}
\address{\sc Simone Marchesi\\
Dipartimento di Matematica "Federigo Enriques"\\
Universit\a degli Studi di Milano\\
via Cesare Saldini  50\\
20133 Milano\\
Italy}
\address{Departamento de \'{A}lgebra\\
Facultad de Ciencias Matem\'{a}ticas\\
Universidad Complutense de Madrid\\
Plaza de las Ciencias  3\\
28040 Madrid\\
Spain}
\email{simone.marchesi@unimi.it, smarches@mat.ucm.es}
\author[Alex Massarenti]{Alex Massarenti}
\address{\sc Alex Massarenti\\
SISSA\\
via Bonomea 265\\
34136 Trieste\\ Italy}
\email{alex.massarenti@sissa.it}

\maketitle

\tableofcontents

\section*{Introduction}
\textit{Bonavero} and \textit{H\"oring}, in \cite{BH}, consider a smooth scheme theoretical complete intersection $X\subset\p^{N}$ and give a bound involving the degrees of the polynomials defining the variety that grants the conic-connectedness. However their result ensures the existence of a smooth conic that in general is weaker than the existence of a singular conic. Indeed from classical arguments of deformations of chains of rational curves we have that a singular conic through two general points on a smooth variety can be deformed into a smooth conic. The existence of a smooth conic $f:\p^{1}\rightarrow X$ through two general points on a projective variety does not imply the existence of a singular conic. This is true if $\dim_{[f]}(\Mor(\p^{1},X;f_{|\{0,\infty\}}))\geq 2$ by Mori's Bend-and-Break lemma \cite[Proposition 3.2]{De}.\\
We start from this result and consider the more general case of a non necessarily smooth variety $X\subset\p^{N}$, set theoretically defined by homogeneous polynomials; such varieties do not need to be a complete intersection. 
In \cite[Theorem 4.4]{MMT} the authors and \textit{Saeed Tafazolian} give a numerical criterion ensuring the existence of a chain of length two of lines through two general points of a variety $X\subset\mathbb{P}^{N}$. 
In section \ref{CL} we generalize \cite[Theorem 4.4]{MMT} and we obtain the following result (Theorem \ref{RC}). 
\begin{theoremn}
Let $X\subset\p^{N}$ be a variety set theoretically defined by homogeneous polynomials $G_{i}$ of degree $d_{i}$, for $i = 1,..,m$, and let $l\geq 2$ be an integer. If
$$\sum_{i=1}^{m}d_{i}\leq \frac{N(l-1)+m}{l}$$
then $X$ is rationally chain connected by chains of lines of length at most $l$.\\ 
In particular if $X$ is smooth and the above inequality is satisfied then $X$ is rationally connected by rational curves of degree at most $l$.
\end{theoremn}
Finally we prove the sharpness of this result considering a hypersurface $X_{l+1}$ of degree $l+1$ in $\mathbb{P}^{l+2}$.

\section{Notation and Preliminaries}

We work over the complex field. Throughout this paper we denote by $X\subset \p^{N}$ an irreducible variety of dimension $n\geq 1$. We assume $X$ to be non-degenerate of codimension $c$, so that $N = n+c$.

\subsubsection*{Prime Fano, covered by lines and conic connected varieties}\label{PF}
Let $x \in X\subset \p^N$ be a general point. We denote by $\mathcal{L}_{x}$ the (possibly empty) variety of lines
through $x$, contained in $X$. Note that $\mathcal{L}_{x}$ is embedded in the space of tangent directions at $x$, that is $\mathcal{L}_{x}\subseteq \p(t_xX^*)=\p^{n-1}$, where $t_xX$ denotes the affine embedded Zariski tangent space at $x$.\\ 
We denote by $a:= \dim(\mathcal{L}_{x})$ the dimension of $\mathcal{L}_{x}$. We say that $X$ is \textit{covered by lines} if $\mathcal{L}_{x}\neq \emptyset$ for $x\in X$ general. When $\mathcal{L}_{x}$ is irreducible, it can proved that $a = \deg(\mathcal{N}_{l/X})$, where $l$ is a line in $X$ through $x$, and $\mathcal{N}_{l/X}$ is its normal bundle. When $a\geq \frac{n-1}{2}$, $\mathcal{L}_x\subset \p^{n-1}$ is smooth and irreducible; if, moreover, $\Pic(X)$ is cyclic, it is also non-degenerate, see \cite{Hwang}.\\
Recall that $X\subset \p^N$ is a \textit{prime Fano variety of index $i(X)$} if its Picard group is generated by the class $H$ of a hyperplane section and $-K_{X}=i(X)H$ for some positive integer $i(X)$. By the work of Mori, see \cite{Mori}, if $i(X) > \frac{n+1}{2}$ then $X$ is covered by lines.\\
A variety $X\subset \p^{N}$ is called a \textit{conic-connected} $(CC)$ variety if for $x,y \in X$ general points there is a conic $C_{x,y}$ passing through $x,y$ and contained in $X$.

\subsubsection*{Loci of Chains}\label{LC}
Let $X$ be a variety covered by lines and let $x\in X$ be a general point. We define the loci determined on $X$ by chains of lines through $x$ as follows.

\begin{Definition}
The locus of lines in $X$ though $x$ is defined as 
$$\mathfrak{Loc}_{1}(x) = \bigcup_{[L]\in \mathcal{L}_{x}}L$$
and the locus of chains of lines of length $l$ in $X$ though $x$ is defined recursively as
$$\mathfrak{Loc}_{l}(x) = \overline{\bigcup_{[L]\in\mathcal{L}_{y} \; | \; y\;\in\;\mathfrak{Loc}_{l-1}(x)\; \rm{is \; a \; general \; point}}L}.$$
\end{Definition} 

We denote by $d_{l}$ the maximal dimension of the irreducible components of $\mathfrak{Loc}_{l}(x)$. If there exists an integer $l$ such that $d_{l} = \dim(X)$ but $d_{l-1} < \dim(X)$ we say that $X$ has \textit{length} $l$. In particular a variety of length $2$ is conic-connected.

In \cite{HK} \textit{Hwang} and \textit{Kebekus} give a lower bound on $d_{l}$ under the irreducibility assumption on $\mathcal{L}_{x}$ for $x\in X$ general point. However their proof work even without this irreducibility assumption. The following theorem can be found in \cite[Theorem 4.6]{Wa}.

\begin{Theorem}\label{Wa}
Let $X$ be a prime Fano variety of dimension $n\geq 3$. Then we have
\begin{center}
$d_{1} = \dim(\mathcal{L}_{x})+1$ and $d_{l} \leq l(\dim(\mathcal{L}_{x})+1)$ for each $l\geq 1$.
\end{center}
\end{Theorem}

In the proof of Theorem \ref{RC} it will be necessary to perform intersections in a product of projective spaces, so we recall briefly the structure of the Chow ring of a product, for more details see \cite{Fu}.

\subsubsection*{Chow ring of a product} 
Let us recall, in order to fix the notation, that for a cartesian product of projective spaces the Chow ring is given by
$$A^{*}(\mathbb{P}^{N_{1}}\times ...\times\mathbb{P}^{N_{k}})\cong\mathbb{Z}[h_{1},...,h_{k}]/(h_{1}^{N_{1}+1},...,h_{k}^{N_{k}+1})$$ 
where $h_{i}$ is the hyperplane class of $\mathbb{P}^{N_{i}}$. It follows that the class of subvariety $Z \subset \mathbb{P}^{N_{1}}\times ...\times\mathbb{P}^{N_{k}}$ can be written in the form
$$[Z]\: = \sum_{i_{1}+...+i_{k} = \dim(Z)}\lambda_{i_{1},...,i_{k}}\,h_{1}^{N_{1}-i_{1}}...h_{k}^{N_{k}-i_{k}},$$
where the $\lambda_{i_{1},...,i_{k}}$ are the multidegrees of $Z$.

\section{Chains of Lines}\label{CL}
We want to generalize \cite[Theorem 4.4]{MMT} giving conditions on the equations defining the variety that ensure the existence of a chain of lines of a prescribed length through two general points of $X$. In this case we have to perform a intersection in a product of projective spaces. 

\begin{Theorem}\label{RC}
Let $X\subset\p^{N}$ be a variety set theoretically defined by homogeneous polynomials $G_{i}$ of degree $d_{i}$, for $i = 1,..,m$, and let $l\geq 2$ be an integer. If
$$\sum_{i=1}^{m}d_{i}\leq \frac{N(l-1)+m}{l}$$
then $X$ is rationally chain connected by chains of lines of length at most $l$.\\ 
In particular if $X$ is smooth and the above inequality is satisfied then $X$ is rationally connected by rational curves of degree at most $l$.
\end{Theorem}
\begin{proof}
Let $x,y \in X$ be two general points. We can assume that $x=[1:0:\ldots:0]$ and $y=[0:\ldots:0:1]$. Moreover let us consider other $l-1$ points in $\p^N$, that we denote by $p^
i=[p^i_0:\ldots:p^i_N]$, for $i=1,\ldots,l-1$. Our goal is to find the set of $l-1$ points that will represent the intersections of the lines we are seeking to build the chain and we would like to see such set as a point of the cartesian product $\p^N_{1} \times \cdots \times \p^N_{l-1}$.\\
Let us consider the line that join the points $x$ and $p^1$, that we will denote by $ux+vp^1=[u+vp^1_0:\ldots:vp^1_N]$. We may observe that $G_i(ux+vp^1)$ is a polynomial of degree $d_i$ in the variables $u$ and $v$; it has $d_i+1$ coefficients, but in this case there is no monomial of the type $u^{d_i}$ because of our assumption $x \in X$. So imposing $G_i(ux+vp^1)\equiv 0$  gives us $d_i$ conditions, involving only on the coordinates of the point $p^1$.\\ 
Let us now consider the line  $up^1+vp^2=[up^1_0+vp^2_0:\ldots:up^1_N+vp^2_N]$. Again \mbox{$G_i(up^1+vp^2)$} is a homogeneous polynomial of degree $d_{i}$ in the variables $u,v$ and imposing $ G_i(up^1+vp^2) \equiv 0$ gives us $d_i + 1$ equations. Let us give a closer look to such equations; $d_i-1$ of them will be homogeneous polynomials in the coordinates of the points $p^1,p^2$, one of them will be a polynomial only in coordinates of the point $p^1$, which is exactly the equation $G_{i} = 0$ written in the coordinates of $p^{1}$ that we have already found in the previous step, and one of them will be a polynomial only in coordinates of the point $p^2$.\\ 
In the same way for every line $up^{i-1}+vp^i=[up^{i-1}_0+vp^i_0:\ldots:up^{i-1}_N+vp^i_N]$, for $i$ from $2$ to $l-1$, imposing $G_i(up^{i-1}+vp^i)\equiv 0$, we get $d_i+1$ equations; $d_i-1$ of them will be homogeneous polynomials in the coordinates of the points $p^{i-1},p^i$, one of them will be a polynomial only in the coordinates of the point $p^{i-1}$, which is the equation $G_{i} = 0$ written in the coordinates of $p^{i-1}$ that we have already found in the previous step, and one of them will be a polynomial only in coordinates of the point $p^i$.\\ 
We now consider the line $up^{l-1} + v y = [up^{l-1}_0:\ldots:up^{l-1}_N + v]$; we notice that $G_i(up^{l-1} + vy)$ is a polynomial of degree $d_i$ in the variables $u$ and $v$, it has $d_i+1$ coefficients, but in this case there is no monomial of the type $v^{d_i}$ because of our assumption $y \in X$. So imposing $G_i(up^{l-1} + vy)\equiv 0$  gives us $d_i$ conditions, only on the coordinates of the point $p^{l-1}$.\\ 
Summarizing, we get the following conditions: 
\begin{itemize}
\item[-] $\sum_{i=1}^m d_i$ homogeneous equations only in the $p^1_j$'s and $\sum_{i=1}^m d_i$ equations only in the $p^{l-1}_j$'s,
\item[-] $m$ homogeneous equations only in the $p^k_j$'s, for every $k=2,\ldots,l-2$,
\item[-] $\sum_{i=1}^m d_i - m$ bihomogeneous equations in the variables $p^{k-1}_j,p^k_j$'s, for every $k=2,\ldots,l-1$.
\end{itemize}
We want to perform the intersection of these hypersurfaces in $\p^{N}_{1}\times ...\times\p^{N}_{l-1}$. If $h_{i}$ is the hyperplane class of $\p^{N}_{i}$ the intersection is given by an expression of the form
$$h_{1}^{\sum_{i=1}^{m}d_{i}}h_{2}^{m}...h_{l-2}^{m}h_{l-1}^{\sum_{i=1}^{m}d_{i}}(h_{1}+h_{2})^{\sum_{i=1}^{m}d_{i}-m}...(h_{l-2}+h_{l-1})^{\sum_{i=1}^{m}d_{i}-m},$$
and each summand of the expression above is of the form
$$
\left.
\begin{array}{l}
h_{1}^{\sum_{i=1}^{m}d_{i}}h_{2}^{m}... h_{l-2}^{m}(h_{1}^{j_{1}}h_{2}^{\sum_{i=1}^{m}d_{i}-m-j_{1}})(h_{2}^{j_{2}}h_{3}^{\sum_{i=1}^{m}d_{i}-m-j_{2}})... (h_{l-2}^{j_{l-2}}h_{l-1}^{\sum_{i=1}^{m}d_{i}-m-j_{l-2}})\\
=h_{1}^{\sum_{i=1}^{m}d_{i}+j_{1}}h_{2}^{\sum_{i=1}^{m}d_{i}-j_{1}+j_{2}}h_{3}^{\sum_{i=1}^{m}d_{i}-j_{2}+j_{3}}... h_{l-2}^{\sum_{i=1}^{m}d_{i}-j_{l-3}+j_{l-2}}h_{l-1}^{2\sum_{i=1}^{m}d_{i}-m-j_{l-2}}.
\end{array}
\right.
$$
Our aim is to prove that under the numerical hypothesis of the theorem at least one of these summands does not vanish. Take 
$$\overline{j_{k}} = \lfloor \frac{k}{l-1}(\sum_{i=1}^{m}d_{i}-m)\rfloor, \:\:\:\mbox{for} \:\:k = 1,...,l-2,$$ 
where $\lfloor p \rfloor$ is the greatest integer smaller or equal than $p$.\\ Note that $0\leq\overline{j_{k}}\leq\sum_{i=1}^{m}d_{i}-m$ for $k = 1,...,l-2$. Consider the term
$$h_{1}^{\sum_{i=1}^{m}d_{i}+\overline{j_{1}}}h_{2}^{\sum_{i=1}^{m}d_{i}-\overline{j_{1}}+\overline{j_{2}}}h_{3}^{\sum_{i=1}^{m}d_{i}-\overline{j_{2}}+\overline{j_{3}}}... h_{l-2}^{\sum_{i=1}^{m}d_{i}-\overline{j_{l}}_{-3}+\overline{j_{l}}_{-2}}h_{l-1}^{2\sum_{i=1}^{m}d_{i}-m-\overline{j_{l}}_{-2}}.$$
For any $k = 2,...,l-2$ the exponent of $h_{k}$ is $\sum_{i=1}^{m}d_{i}-\overline{j_{k}}_{-1}+\overline{j_{k}}$. In order to ensure the intersection in $\p^{N}_{k}$ to be not empty we impose
\begin{center}
$N-\sum_{i=1}^{m}d_{i}+\overline{j_{k}}_{-1}-\overline{j_{k}}\geq0$ for any $k = 2,...,l-2$.
\end{center}
Substituting we have 
$$\lfloor\frac{k-1}{l-1}(\sum_{i=1}^{m}d_{i}-m)\rfloor\geq\lfloor\frac{k}{l-1}(\sum_{i=1}^{m}d_{i}-m)\rfloor -N+\sum_{i=1}^{m}d_{i}.$$ 
Since the number on the right is an integer it is enough to prove that $\frac{k-1}{l-1}(\sum_{i=1}^{m}d_{i}-m)\geq\lfloor\frac{k}{l-1}(\sum_{i=1}^{m}d_{i}-m)\rfloor -N+\sum_{i=1}^{m}d_{i}$ that is 
$$\frac{k-1}{l-1}(\sum_{i=1}^{m}d_{i}-m)+N-\sum_{i=1}^{m}d_{i}\geq \lfloor\frac{k}{l-1}(\sum_{i=1}^{m}d_{i}-m)\rfloor.$$ 
Consider now the term on the left, from our hypothesis we get $N\geq\frac{l\sum_{i=1}^{m}d_{i}-m}{l-1}$. So
$$
\left.
\begin{array}{l}
\frac{k-1}{l-1}(\sum_{i=1}^{m}d_{i}-m)+N-\sum_{i=1}^{m}d_{i}\geq\frac{k-1}{l-1}(\sum_{i=1}^{m}d_{i}-m)+\frac{l\sum_{i=1}^{m}d_{i}-m}{l-1}-\sum_{i=1}^{m}d_{i}\\
=\frac{k}{l-1}(\sum_{i=1}^{m}d_{i}-m)\geq\lfloor\frac{k}{l-1}(\sum_{i=1}^{m}d_{i}-m)\rfloor.
\end{array}
\right.
$$
This prove that the intersection in $\p^{N}_{k}$ is not empty for any $k = 2,...,l-2$.\\
Consider now the exponent of $h_{1}$. We have to impose
$$N-\sum_{i=1}^{m}d_{i}-\overline{j_{1}} = N-\sum_{i=1}^{m}d_{i}-\lfloor\frac{\sum_{i=1}^{m}d_{i}-m}{l-1}\rfloor\geq 0.$$
But $N-\sum_{i=1}^{m}d_{i}\geq\frac{\sum_{i=1}^{m}d_{i}-m}{l-1}$ implies the last inequality and is equivalent to the numerical hypothesis of the theorem. This show that the intersection in $\p^{N}_{1}$ is also not empty.\\
Finally consider the exponent of $h_{l-1}$ and impose
$$N-2\sum_{i=1}^{m}d_{i}+m+\overline{j_{l}}_{-2}\geq 0.$$
We have $\lfloor\frac{l-2}{l-1}(\sum_{i=1}^{m}d_{i}-m)\rfloor\geq 2\sum_{i=1}^{m}d_{i}-m-N$, since the number on the right is an integer it is enough to prove that $(l-2)(\sum_{i=1}^{m}d_{i}-m)\geq (l-1)(2\sum_{i=1}^{m}d_{i}-m-N)$ and again this is exactly our numerical hypothesis. We conclude that the intersection in $\p^{N}_{l-1}$ is not empty.\\
At this point we know that the equations define non-empty subvarieties of $\p^{N}_{j}$ for any $j = 1,...,l-1$. To ensure that these subvarieties lift to subvarieties of $\p^{N}_{1}\times ...\times\p^{N}_{l-1}$ rationally equivalent to cycles having non-empty intersection we force 
$$N(l-1)-(l-2)\sum_{i=1}^{m}d_{i}-2\sum_{i=1}^{m}d_{i}+m\geq 0,$$
which once again is our initial hypothesis.\\
So our system of equations on the product $\p^{N}_{1}\times ...\times\p^{N}_{l-1}$ has at least a solution, which represents the sequence of connection points in the chain of lines we were looking for. Finally, if $X$ is smooth, by general smoothing arguments (\cite{De}, Proposition 4.24) we can deform our chain in a rational curve of degree at most $l$ connecting $x$ and $y$.
\end{proof}

\subsection{Sharpness of Theorem \ref{RC}}\label{sha} The inequality in Theorem \ref{RC} is sharp. Consider a smooth hypersurface $X_{l+1}$ of degree $l+1$ in $\mathbb{P}^{l+2}$. Then $d = l+1$, $N = l+2$ and $m = 1$, so we have $\frac{N(l-1)+m}{l} = \frac{l^{2}+l-1}{l}$. Since
$$l\leq \frac{l^{2}+l-1}{l} < l+1$$
the hypersurface $X_{l+1}$ is a good candidate to prove sharpness. The equalities imply $d = l+1 = l+2-1 = N-1$ we have $\dim(\mathcal{L}_{x}) = 0$. Now by Theorem \ref{Wa} the dimension $d_{l} = \dim(\mathfrak{Loc}_{l}(x))$ is bounded by 
$d_{l}\leq l(\dim(\mathcal{L}_{x})+1) = l$. Since $\dim(X_{l+1}) = l+1 > l$ we have
$$\dim(\mathfrak{Loc}_{l}(x)) < \dim(X)$$
and $X$ is not connected by chains of length $l$ of lines.\\ 

From Theorem \ref{RC} we get the following Corollary which can also be found in \cite[Lemma 4.8.1]{Ko}.

\begin{Corollary}
Let $X\subset\p^{N}$ be a hypersurface of degree $d\leq N-1$. Then $X$ is rationally chain connected by a chain of lines of length at most $N-1$.
\end{Corollary}
\begin{proof}
If $l = N-1$ the inequality $d\leq\frac{N(N-2)+1}{N-1} = \frac{N^{2}-2N+1}{N-1} = N-1$ of Theorem \ref{RC} is satisfied by hypothesis.
\end{proof}

\begin{Corollary}
Let $X\subset\p^{N}$ be a scheme theoretical complete intersection. If $\deg(X)\leq\frac{N(l-1)+c}{l}$
then $X$ is rationally chain connected by chains of lines of length at most $l$. 
If, in addition, $X$ is smooth and Fano of index $i_{X}\geq\frac{n+l}{l}$ then $X$ is rationally connected by rational curves of degree at most $l$.
\end{Corollary}
\begin{proof}
The first assertion follows from the inequality $\sum_{i=1}^{c}d_{i}\leq\prod_{i=1}^{c}d_{i} = \deg(X)$ combined with Theorem \ref{RC}. If $X$ is a smooth, Fano, complete intersection, Theorem \ref{RC} and the equality $i_{X} = N+1-\sum_{i=1}^{c}d_{i}$ imply the second assertion.
\end{proof}

\begin{Proposition}
Let $X\subseteq\p^{N}$ be a smooth complete intersection defined by homogeneous polynomials $G_{i}$ of degree $d_{i}$, for $i = 1,..,c$, such that $\sum_{i=1}^{c}d_{i}\leq N-1$. Then
$$length(X) = \lceil \frac{N-c}{N-\sum_{i=1}^{c}d_{i}}\rceil,$$
where $\lceil k\rceil$ is the smallest integer greater or equal than $k$.
\end{Proposition}
\begin{proof}
Since the integer $l_{min} := \lceil \frac{N-c}{N-\sum_{i=1}^{c}d_{i}}\rceil$ satisfies the inequality of Theorem \ref{RC} we have that $X$ is $l_{min}$-chain connected. We have to prove that $X$ is not $(l_{min}-1)$-chain connected. Now $l_{min}-1 = \lceil \frac{N-c}{N-\sum_{i=1}^{c}d_{i}}-1\rceil$. Note that $\dim(\mathcal{L}_{x}) = N-\sum_{i=1}^{c}d_{i}-1\geq 0$. By Theorem \ref{Wa} we have $d_{l_{ min}-1}\leq (l_{min}-1)(\dim(\mathcal{L}_{x})+1)$, we distinguish two cases
\begin{itemize}
\item[-] If $\frac{N-c}{N-\sum_{i=1}^{c}d_{i}}-1$ is an integer then
$$d_{l_{ min}-1}\leq (\frac{N-c}{N-\sum_{i=1}^{c}d_{i}}-1)(N-\sum_{i=1}^{c}d_{i}) = \sum_{i=1}^{c}d_{i}-c < N-c = n.$$
\item[-] If $\frac{N-c}{N-\sum_{i=1}^{c}d_{i}}-1$ is not an integer then 
$$d_{l_{ min}-1}\leq (l_{min}-1)(N-\sum_{i=1}^{c}d_{i}) < \frac{N-c}{N-\sum_{i=1}^{c}d_{i}}(N-\sum_{i=1}^{c}d_{i}) = n.$$
\end{itemize} 
Then $d_{l_{min}-1} = \dim(\mathfrak{Loc_{l_{min}-1}}(x)) < \dim(X)$.
\end{proof}

\begin{Remark}
In the case $l = 2$ we find again \cite[Theorem 4.4]{MMT}, in fact, the inequality in \ref{RC} simply becomes $\sum_{i=1}^{m}d_{i}\leq \frac{N+m}{2}$.
\end{Remark}

\begin{Remark}
In the range of Theorem \ref{RC} $X$ is covered by lines. Indeed under the numerical hypothesis of Theorem \ref{RC} we have $m < \sum_{i=1}^{m}d_{i}\leq \frac{N(l-1)+m}{l}$ which gives $m < N$. So we get the inequality   
$$\sum_{i=1}^{m}d_{i}\leq \frac{N(l-1)+m}{l}< N,$$
which forces $X$ to be covered by lines.
\end{Remark}

\subsubsection*{Counting chains of lines}
We discuss now an example that shows how it is possible to count the number of possible chains of lines when the equality in Theorem \ref{RC} holds.\\
Let us consider a cubic threefold $X\subset\p^{4}$. In this case the equality holds when $l=3$, so we are looking for all the possible $3$-chains of lines connecting two general points $x,y$ of $X$. Following the proof of the theorem, we have to perform intersection in $\p^{4}_{1}\times\p^{4}_{2}$, we are looking for two points $p^{1},p^{2}$. We have $3$ conditions on the coordinates of $p^{1}$ namely $h_{1},2h_{1},3h_{1}$, describing the cone of lines in $X$ through $x$. Furthermore we have $3$ other conditions $h_{2},2h_{2},3h_{2}$ on the coordinates of $p^{2}$, describing the cone of lines in $X$ through $y$. Finally we have $2$ conditions involving the coordinates of both points $p^{1},p^{2}$, namely $2h_{1}+h_{2}$ and $h_{1}+2h_{2}$. Their intersection product is given by
$$h_{1}2h_{1}3h_{1}h_{2}2h_{2}3h_{2}(2h_{1}+h_{2})(h_{1}+2h_{2}) = 36h_{1}^{3}h_{2}^{3}(2h_{1}^{2}+5h_{1}h_{2}+2h_{2}^{2}) = 180h_{1}^{4}h_{2}^{4},$$
and we conclude that we have $180$ possibilities.\\
A geometrical description of this fact is the following: there are exactly $6 = h_{1}2h_{1}3h_{1}$ lines in $X$ through $x$ and $6 = h_{2}2h_{2}3h_{2}$ lines in $X$ through $y$. Take a line $L_{x}$ of the first family and a line $L_{y}$ of the second. These lines are skew otherwise $X$ would be connected by singular conics and we know this is not possible by classical arguments of projective geometry, so $L_{x}$ and $L_{y}$ generate a $3$-plane $H$. The linear section $H\cap X:=S$ is a smooth cubic surface in $\p^{3}$ and we can consider the lines $L_{x},L_{y}$ as two exceptional divisors of a proper blow-up of the projective plane in $6$ points; we denote by $p,q$ the points in the plane which are blown-up in $L_{x}$ and $L_{y}$. There are exactly $5$ lines joining $L_{x}$ and $L_{y}$ namely the strict transform of the line $\langle p,q\rangle$ and of the conics passing through $p,q$ and $3$ of the $4$ remaining points. In conclusion we have $6\cdot 6\cdot 5 = 180$ possibilities, that double-checks the counting made before.              

\subsubsection*{Acknowledgements}
The authors heartily thank \textit{Prof. Paltin Ionescu} and \textit{Dr. Jos$\acute{e}$ Carlos Sierra} for the introduction to the subject, many helpful comments and suggestions. We would like to give a special thank also to \textit{Prof. Enrique Arrondo} and \textit{Prof. Massimiliano Mella}.

\end{document}